\documentclass[a4paper,12pt,reqno]{amsart}
\usepackage{amsfonts}
\usepackage{amsmath}
\usepackage{amssymb}
\usepackage[a4paper]{geometry}
\usepackage{mathrsfs}
\usepackage{csquotes}
\usepackage{enumitem}
\usepackage[colorlinks]{hyperref}
\renewcommand\eqref[1]{(\ref{#1})} 
\numberwithin{equation}{section}
\theoremstyle{plain}
\newtheorem{thm}{Theorem}[section]

\newtheorem{lem}[thm]{Lemma}
 
\theoremstyle{definition}
\newtheorem{defn}[thm]{Definition}
\newtheorem{rem}[thm]{Remark}

\newcommand{\Rn}{\mathbb R^{n}}

\def\R{\mathcal R}

\def\e[#1]{{\textrm{e}}^{#1}}

\def\Rn{{\mathbb R}^n}
\def\G{{\mathbb G}}

\begin{document}

   \title[A comparison principle for higher order nonlinear hypoelliptic]
   {A comparison principle for higher order nonlinear hypoelliptic heat operators \\on graded Lie groups}

\author[M. Ruzhansky]{Michael Ruzhansky}
\address{
  Michael Ruzhansky:
 \endgraf
    Department of Mathematics: Analysis, Logic and Discrete Mathematics
  \endgraf
    Ghent University, Belgium
   \endgraf
  and
  \endgraf
  School of Mathematical Sciences
    \endgraf
    Queen Mary University of London, United Kingdom
   \endgraf
  {\it E-mail address} {\rm
michael.ruzhansky@ugent.be}
  }

\author[N. Yessirkegenov]{Nurgissa Yessirkegenov}
\address{
  Nurgissa Yessirkegenov:
  \endgraf
  Department of Mathematics: Analysis, Logic and Discrete Mathematics
  \endgraf
    Ghent University, Belgium
  \endgraf
    and Department of Mathematics, Imperial College London, United Kingdom
   \endgraf
  and Institute of Mathematics and Mathematical Modelling, Kazakhstan
    \endgraf
  {\it E-mail address} {\rm nurgissa.yessirkegenov@gmail.com}
  }

\thanks{MR was supported by the EPSRC Grant EP/R003025/1, by the Leverhulme Research
Grant RPG-2017-151, and by the FWO Odysseus 1 grant G.0H94.18N: Analysis and Partial Differential Equations. NY was supported by the FWO Odysseus 1 grant G.0H94.18N: Analysis and Partial Differential Equations and by the LMS Early Career Fellowship grant ECF-1819-12.}


     \begin{abstract} In this paper we present a comparison principle for higher order nonlinear hypoelliptic heat operators on graded Lie groups. Moreover, using the comparison principle we obtain blow-up type results and global in $t$-boundedness of solutions of nonlinear equations for the heat $p$-sub-Laplacian on stratified Lie groups. In particular, this paper generalises and extends previous results obtained by the first author and Suragan in \cite{RS18}.
     \end{abstract}

  \maketitle

\section{Introduction}
\label{SEC:intro}

A connected simply connected Lie group $\G$ is called a graded Lie group if its Lie algebra admits a gradation. The graded Lie groups form the subclass of homogeneous nilpotent Lie groups admitting homogeneous hypoelliptic left-invariant differential operators (\cite{Miller:80}, \cite{tER:97}, see also a discussion in \cite[Section 4.1]{FR16}). These operators are called Rockland operators from the Rockland conjecture, solved by Helffer and Nourrigat \cite{HN-79}. So, we understand by a Rockland operator {\em any left-invariant homogeneous hypoelliptic differential operator on} $\G$. Thus, the considered setting includes the higher order operators on $\Rn$ as well as higher order hypoelliptic invariant differential operators on the Heisenberg group, on general stratified Lie groups, and on general graded Lie groups.

Let us also recall that the standard Lebesgue measure is the Haar measure for $\G$. Let $\Omega \subset \G$ be a bounded set with smooth boundary. We denote the Sobolev space by $S^{a,p}(\Omega)=S^{a,p}_{\R}(\Omega)$, for $a>0$ and $p\in (1,\infty)\cup \{\infty_{o}\}$, defined by the norm
\begin{equation}\label{Sob_norm}\|u\|_{S^{a, p}(\Omega)}:=\left(\int_{\Omega}\left(\left|\mathcal{R}^{\frac{a}{\nu}} u(x)\right|^{p}+|u(x)|^{p}\right) d x\right)^{\frac{1}{p}},\end{equation}
where $\nu$ is the homogeneous order of the Rockland operator $\R$. We have allowed ourselves to write $\|\cdot\|_{L^{\infty}(\G)}=\|\cdot\|_{L^{\infty_{o}}(\G)}$ for the supremum norm, in the notation of \cite[Chapter 4]{FR16}. Let us also define the functional class $S^{a,p}_{0}(\Omega)$ to be the completion of $C^{\infty}_{0}(\Omega)$ in the norm \eqref{Sob_norm}. For a general discussion of Sobolev spaces on graded Lie groups we refer to \cite[Chapter 4]{FR16} and \cite{FR17}.

In this paper we study the higher order nonlinear hypoelliptic heat equation for $u=u(t,x)$,
\begin{equation}\label{eq_1}
u_{t}-\sum_{j=1}^{n_{2}}\mathcal{R}_{1}^{\frac{a_{1}}{\nu_{1}}}\left(\left|\mathcal{R}_{1}^{\frac{a_{1}}{\nu_{1}}} u\right|^{p_{j}-2} \mathcal{R}_{1}^{\frac{a_{1}}{\nu_{1}}} u\right)u=\alpha \sum_{i=1}^{n_{1}}|u|^{q_{i}-1}u+\beta \sum_{j=1}^{n_{2}}|\R_{2}^{\frac{a_{2}}{\nu_{2}}} u|^{r_{j}}+\gamma \sum_{k=1}^{n_{3}}|u|^{s_{k}-1}u
\end{equation}
for $x \in \Omega$ and $t>0$, with the initial-boundary conditions
\begin{equation}\label{eq_2}
u=0, \quad x \in \partial \Omega, \quad t>0,
\end{equation}
\begin{equation}\label{eq_3}
u(0, x)=u_{0}(x), \quad x \in \Omega,
\end{equation}
where $a_{1},a_{2}\geq 0$, and $\alpha,\beta,\gamma\in \mathbb{R}$, and $n_{1},n_{2}, n_{3}\in \mathbb{N}$, $p_{j}>1$ and
\begin{equation}
    \label{param_range}
q_{i}\left\{\begin{array}{l}
{\geq 1, \quad \text{if} \quad \alpha>0,} \\
{>0, \quad \text{if} \quad \alpha<0,}
\end{array}\right.
s_{k}\left\{\begin{array}{l}
{\geq 1, \quad \text{if} \quad \gamma>0,} \\
{>0, \quad \text{if} \quad \gamma<0,}
\end{array}\right.
r_{j}\left\{\begin{array}{l}
{>1, \quad \text{if} \quad \beta>0,} \\
{>0, \quad \text{if} \quad \beta<0.}
\end{array}\right.
\end{equation}
Here, $\nu_{1}$ and $\nu_{2}$ are the homogeneous orders of the Rockland operators $\R_{1}$ and $\R_{2}$, respectively. We also assume that the initial data satisfies
$$u_{0}\in S_{0}^{a,\infty}(\Omega),\quad u_{0}\geq 0,$$
where $a=\max \{a_{1},a_{2}\}$.

\begin{defn}\label{def_weak} Set $Q_{T}=(0, T) \times \Omega$, $S_{T}=(0, T)\times \partial \Omega$, $\partial Q_{T}=S_{T} \cup\{\{0\}\times\bar{\Omega} \}$, $p=\max\{p_{j}\}$ and $m=\max \{p_{j}, q_{i}, r_{j}, s_{k}\}$. A nonnegative function $u(t, x)$ is called a weak super- (sub-) solution of \eqref{eq_1}-\eqref{eq_3} on $Q_{T}$ if it satisfies
$$
\begin{aligned}
u \in C([0, T) \times\overline{\Omega}) \cap L^{m}\left((0, T) ; S_{0}^{a, m}(\Omega)\right), &\; \partial_{t} u \in L^{2}\left((0, T) ; L^{2}(\Omega)\right), \\
u(0, x) \geq (\leq) u_{0}, &\left.\;u\right|_{\partial \Omega} \geq(\leq) 0,
\end{aligned}
$$
\begin{multline*}
\iint_{Q_{T}} \partial_{t} u \phi+\sum_{j=1}^{n_{2}}|\R^{\frac{a_{1}}{\nu_{1}}}_{1}u|^{p_{j}-2} \R^{\frac{a_{1}}{\nu_{1}}}_{1} u \cdot \R^{\frac{a_{1}}{\nu_{1}}}_{1} \phi dxdt \\ \geq (\leq) \iint_{Q_{T}}\left(\alpha \sum_{i=1}^{n_{1}}|u|^{q_{i}-1}u+\beta \sum_{j=1}^{n_{2}}|\R_{2}^{\frac{a_{2}}{\nu_{2}}} u|^{r_{j}}+\gamma \sum_{k=1}^{n_{3}}|u|^{s_{k}-1}u\right) \phi dxdt,
\end{multline*}
for all $\phi \in C(\overline{Q_{T}}) \cap L^{p}\left((0, T); S_{0}^{a_{1}, p}(\Omega)\right)$ such that $\phi \geq 0,\left.\phi\right|_{S_{T}}=0$. Then $u$ is called a weak solution if it is a super-solution and a sub-solution. Here and after, we use $T_{\max}$ to denote the maximal existence time.
\end{defn}

Our goal in this paper is to give a simple proof of a comparison principle for the initial boundary value problem for higher order nonlinear hypoelliptic heat operators on graded Lie groups using pure algebraic relations, inspired by the works \cite{Att12} and \cite{ZL13}.

The structure of this paper is as follows. Section \ref{SEC:compar} establishes a comparison principle for the problem \eqref{eq_1}-\eqref{eq_3}. Then, in Section \ref{SEC:appl}, using the comparison principle, we investigate the blow-up or the boundedness of solution of \eqref{eq_1}-\eqref{eq_3} depending on the signs of $\alpha$, $\beta$, $\gamma$, and relations between parameters $p_{j}$, $q_{i}$, $r_{j}$, $s_{k}$, and on $u_{0}$.

\section{A comparison principle on graded Lie groups}
\label{SEC:compar}
In this section we state a comparison principle for the problem \eqref{eq_1}-\eqref{eq_3}.
\begin{thm}\label{compar_thm} Assume that $u, v \in L_{\text {loc }}^{\infty}\left((0, T) ; S^{a, \infty}(\Omega)\right)$ are sub- and super-solutions of \eqref{eq_1}-\eqref{eq_3}, respectively. Assume also that at least one of the parameters $\alpha$, $\beta$ and $\gamma$ be positive or $\alpha=\beta=\gamma=0$. Let $r_{j} \geq \frac{p_{j}}{2}$ if $\beta>0$. Then we have $u \leq v$ on $Q_{T}$.
\end{thm}
\begin{rem}
 In the special case $n_{1}=n_{2}=n_{3}=1$, $\beta=0$ and $\alpha \gamma \leq 0$, Theorem \ref{compar_thm} was obtained in \cite{RS18}.
\end{rem}
The proof of the comparison principle mostly based on the following algebraic lemma (see e.g. \cite[Lemma 2.1]{Att12}).
\begin{lem}\label{lem_algeb}
Let $\sigma>1$. For all $\vec{a}$, $\vec{b} \in \mathbb{R}^{N}$, we have
$$
\left\langle|\vec{a}|^{\sigma-2} \vec{a}-|\vec{b}|^{\sigma-2} \vec{b}, \vec{a}-\vec{b}\right\rangle \geq\frac{4}{\sigma^{2}}|| \vec{a}|^{\frac{\sigma-2}{2}} \vec{a}-\left.|\vec{b}|^{\frac{\sigma-2}{2}} \vec{b}\right|^{2}.
$$
\end{lem}
\begin{proof}[Proof of Theorem \ref{compar_thm}] First, let us consider the case $\alpha,\beta$ and $\gamma>0$. Denote $\phi:=\max \{u-v, 0\}$, hence $\phi(0, x)=0$ and $\left.\phi(t, x)\right|_{x \in \partial \Omega}=0$. By the definitions of sub- and super-solutions, using $\phi$ as the test function, for any $\tau \in(0, T)$, we have
\begin{equation}
    \label{I0}
\begin{aligned}\frac{1}{2}& \int_{\Omega} \phi^{2}(\tau, x) d x=\int_{0}^{\tau} \int_{\Omega} \frac{1}{2} \partial_{t}\left(\phi^{2}(t, x)\right) d x d t = \int_{0}^{\tau} \int_{\Omega} \partial_{t} \phi \phi dxdt\\& \leq - \underbrace{\sum_{j=1}^{n_{2}}\int_{0}^{\tau} \int_{\{\phi(t, \cdot)>0\}}\left(|\mathcal{R}_{1}^{\frac{a_{1}}{\nu_{1}}} u|^{p_{j}-2} \mathcal{R}_{1}^{\frac{a_{1}}{\nu_{1}}} u-|\mathcal{R}_{1}^{\frac{a_{1}}{\nu_{1}}} v|^{p_{j}-2} \mathcal{R}_{1}^{\frac{a_{1}}{\nu_{1}}} v\right) \cdot \mathcal{R}_{1}^{\frac{a_{1}}{\nu_{1}}} \phi dxdt}_{I_{1}} \\ &+\beta\underbrace{\sum_{j=1}^{n_{2}} \int_{0}^{\tau} \int_{\{\phi(t, \cdot)>0\}} \left(|\R_{2}^{\frac{a_{2}}{\nu_{2}}} u|^{r_{j}}-|\R_{2}^{\frac{a_{2}}{\nu_{2}}} v|^{r_{j}}|\right) \phi dxdt}_{I_{2}} \end{aligned}\end{equation}
$$
+\alpha\underbrace{\sum_{i=1}^{n_{1}}\int_{0}^{\tau} \int_{\{\phi(t, \cdot)>0\}}\left(|u|^{q_{i}-1}u-|v|^{q_{i}-1}v\right) \phi dxdt}_{I_{3}}
$$
$$
+\gamma\underbrace{\sum_{k=1}^{n_{3}}\int_{0}^{\tau} \int_{\{\phi(t, \cdot)>0\}}\left(|u|^{s_{k}-1}u-|v|^{s_{k}-1}v\right) \phi dxdt}_{I_{4}}.
$$
By Lemma \ref{lem_algeb}, for $I_{1}$ we have
\begin{equation}
    \label{I1}
    I_{1}\geq  \sum_{j=1}^{n_{2}}\frac{4}{p_{j}^{2}}\int_{0}^{\tau} \int_{\{\phi(t, \cdot)>0\}}\left||\mathcal{R}_{1}^{\frac{a_{1}}{\nu_{1}}} u|^{\frac{p_{j}-2}{2}} \mathcal{R}_{1}^{\frac{a_{1}}{\nu_{1}}} u-|\mathcal{R}_{1}^{\frac{a_{1}}{\nu_{1}}} v|^{\frac{p_{j}-2}{2}} \mathcal{R}_{1}^{\frac{a_{1}}{\nu_{1}}} v\right|^{2} dxdt.
\end{equation}
Let us now estimate the term $I_{2}$. We put $h(s)=s^{\frac{2r_{j}}{p_{j}}}$ for $s \geqslant 0$. Given that $r_{j}  \geqslant \frac{p_{j}}{2}$, we have $h^{\prime}(s)=\frac{2 r_{j}}{p_{j}} s^{\frac{2 r_{j}-p_{j}}{p_{j}}}$. Then, by the mean value theorem we have
$$
\left||\R_{2}^{\frac{a_{2}}{\nu_{2}}} u|^{r_{j}}-|\R_{2}^{\frac{a_{2}}{\nu_{2}}} v|^{r_{j}}\right|^{2} \leq C h^{\prime}(\theta)^{2}\left||\R_{2}^{\frac{a_{2}}{\nu_{2}}} u|^{p_{j}/2}-|\R_{2}^{\frac{a_{2}}{\nu_{2}}} v|^{p_{j}/2}\right|^{2},
$$
for some $0 \leqslant \theta \leqslant \max \left(|\R_{2}^{\frac{a_{2}}{\nu_{2}}} u|^{p_{j}/2},|\R_{2}^{\frac{a_{2}}{\nu_{2}}} v|^{p_{j}/2}\right)$.

A direct computation yields that
$$
\left||\R_{2}^{\frac{a_{2}}{\nu_{2}}} u|^{p_{j}/2}-|\R_{2}^{\frac{a_{2}}{\nu_{2}}} v|^{p_{j}/2}\right|^{2} \leqslant\left||\R_{2}^{\frac{a_{2}}{\nu_{2}}} u|^{(p_{j}-2)/2}\R_{2}^{\frac{a_{2}}{\nu_{2}}} u-|\R_{2}^{\frac{a_{2}}{\nu_{2}}} v|^{(p_{j}-2)/2}\R_{2}^{\frac{a_{2}}{\nu_{2}}} v\right|^{2}.
$$
Taking into account $u, v \in L^{\infty}\left((0, T); S^{a, \infty}(\Omega)\right)$, it follows that
\begin{equation}
    \label{combin1}
\left||\R_{2}^{\frac{a_{2}}{\nu_{2}}} u|^{r_{j}}-|\R_{2}^{\frac{a_{2}}{\nu_{2}}} v|^{r_{j}}\right|^{2} \leqslant C\left||\R_{2}^{\frac{a_{2}}{\nu_{2}}} u|^{(p_{j}-2)/2}\R_{2}^{\frac{a_{2}}{\nu_{2}}} u-|\R_{2}^{\frac{a_{2}}{\nu_{2}}} v|^{(p_{j}-2)/2}\R_{2}^{\frac{a_{2}}{\nu_{2}}} v\right|^{2},
\end{equation}
where $C$ is a positive constant depending on $r_{j},p_{j}$ and $\max \{|\R_{2}^{\frac{a_{2}}{\nu_{2}}} u|^{p_{j}/2},|\R_{2}^{\frac{a_{2}}{\nu_{2}}} v|^{p_{j}/2}\}$.

On the other hand, by Young's inequality we have
\begin{multline}
\label{combin2}
I_{2} \leqslant \sum_{j=1}^{n_{2}}\epsilon_{j}\int_{0}^{\tau} \int_{\{\phi(t, \cdot)>0\}} \left||\R_{2}^{\frac{a_{2}}{\nu_{2}}} u|^{r_{j}}-|\R_{2}^{\frac{a_{2}}{\nu_{2}}} v|^{r_{j}}\right|^{2}d x d t\\+\sum_{j=1}^{n_{2}}C(\epsilon_{j}) \int_{0}^{\tau} \int_{\{\phi(t, \cdot)>0\}} \phi^{2} d x d t.
\end{multline}
A combination of \eqref{combin1} and \eqref{combin2} leads to
\begin{multline}\label{I2}
I_{2} \leqslant C  \sum_{j=1}^{n_{2}}\epsilon_{j}\int_{0}^{\tau} \int_{\{\phi(t, \cdot)>0\}} \left||\R_{2}^{\frac{a_{2}}{\nu_{2}}} u|^{(p_{j}-2)/2}\R_{2}^{\frac{a_{2}}{\nu_{2}}} u-|\R_{2}^{\frac{a_{2}}{\nu_{2}}} v|^{(p_{j}-2)/2}\R_{2}^{\frac{a_{2}}{\nu_{2}}} v\right|^{2}d x d t\\+\sum_{j=1}^{n_{2}}C(\epsilon_{j}) \int_{0}^{\tau} \int_{\{\phi(t, \cdot)>0\}} \phi^{2} d x d t.
\end{multline}
For $I_{3}$, by the mean value theorem we obtain
\begin{equation}
    \label{I3}
I_{3} \leq \sum_{i=1}^{n_{1}}q_{i}\|u\|_{L^{\infty}}^{q_{i}-1} \int_{0}^{\tau} \int_{\{\phi(t, \cdot)>0\}} \phi^{2} dxdt.
\end{equation}
Similarly, for $I_{4}$ we have
\begin{equation}
    \label{I4_1}
I_{4} \leq \sum_{k=1}^{n_{3}}s_{k}\|u\|_{L^{\infty}}^{s_{k}-1} \int_{0}^{\tau} \int_{\{\phi(t, \cdot)>0\}} \phi^{2} dxdt.
\end{equation}
Choosing $0<\epsilon_{j}<4/(\beta C p_{j}^{2})$ and combining the estimates \eqref{I0}, \eqref{I1}, \eqref{I2}, \eqref{I3} and \eqref{I4_1}, we obtain for any $\tau \in (0,T)$ that
\begin{multline}\label{final_est}
\int_{\Omega} \phi^{2}(\tau) dx \\ \leq C\left(\alpha, \beta, \epsilon, q_{i}, s_{k}, r_{j}, p_{j},\|u\|_{L^{\infty}}, \max \{|\R_{2}^{\frac{a_{2}}{\nu_{2}}} u|^{p_{j}/2},|\R_{2}^{\frac{a_{2}}{\nu_{2}}} v|^{p_{j}/2}\}\right) \int_{0}^{\tau} \int_{\Omega} \phi^{2} dxdt.
\end{multline}
Then by Gronwall's lemma we conclude that $\phi\equiv0$ almost everywhere.

The case $\alpha=\beta=\gamma=0$ is trivial.

Now, we discuss the case, when not all, but at least one of the parameters $\alpha$, $\beta$, $\gamma$ is positive. Note that $I_{3}$ is positive, since for $q_{i}>0$ we have
$$
\left\{\begin{array}{l}
{|u|^{q_{i}-1} u-|v|^{q_{i}-1} v=u^{q_{i}}-v^{q_{i}}>0, \quad \text{for} \quad u>v>0} \\
{|u|^{q_{i}-1} u-|v|^{q_{i}-1} v=u^{q_{i}}+|v|^{q_{i}}>0, \quad \text{for} \quad u>0>v} \\
{|u|^{q_{i}-1} u-|v|^{q_{i}-1} v=-|u|^{q_{i}}+|v|^{q_{i}}>0, \quad \text{for} \quad 0>u>v}.
\end{array}\right.
$$
Similarly, one can verify that $I_{4}$ is positive for $s_{k}>0$. Therefore, in the case when $\alpha<0$ or $\beta<0$ (or $\gamma<0$) by dropping $I_{3}$ or $I_{2}$ (or $I_{4}$), respectively, we can always get \eqref{final_est}.
\end{proof}

\section{Some applications to nonlinear equations for the heat $p$-sub-Laplacian}
\label{SEC:appl} In this section, we give some applications of Theorem \ref{compar_thm} to nonlinear equations for the heat $p$-sub-Laplacian on stratified Lie groups. These groups are an important class of graded Lie groups, investigated thoroughly by Folland \cite{Fol75}. There are many different, equivalent ways to define a stratified Lie group (see, for example, \cite{BLU07, FS-book} or \cite{FR16, RS-book} for the Lie group and Lie algebra points of view, respectively). A Lie group $\mathbb{G}=\left(\mathbb{R}^{N}, \circ\right)$ is called a stratified Lie group if it satisfies the following two conditions:
\begin{itemize}
\item for every $\lambda>0$ the dilation $\delta_{\lambda}: \mathbb{R}^{N} \rightarrow \mathbb{R}^{N}$ defined by
$$
\delta_{\lambda}(x) \equiv \delta_{\lambda}\left(x^{'}, \ldots, x^{(r)}\right):=\left(\lambda x^{'}, \ldots, \lambda^{r} x^{(r)}\right)
$$
is an automorphism of the group $\mathbb{G}$, where $x^{\prime} \equiv x^{(1)} \in \mathbb{R}^{N_{1}}$ and $x^{(k)} \in \mathbb{R}^{N_{k}}$ for $k=2, \ldots, r$ with $N_{1}+\cdots+N_{r}=N$ and $\mathbb{R}^{N}=\mathbb{R}^{N_{1}} \times \cdots \times \mathbb{R}^{N_{r}}$.
\item let $X_{1}, \ldots, X_{N_{1}}$ be the left invariant vector fields on $\mathbb{G}$ such that $X_{j}(0)=\frac{\partial}{\partial x_{j}}|_{0}$ for $j=1, \ldots, N_{1}$. Then, for every $x \in \mathbb{R}^{N}$ the H\"{o}rmander condition
$$
\operatorname{rank}\left(\operatorname{Lie}\left\{X_{1}, \ldots, X_{N_{1}}\right\}\right)=N
$$
holds, that is, $X_{1}, \ldots, X_{N_{1}}$ with their iterated commutators span the whole Lie algebra of the group $\G$.
\end{itemize}
Let us also recall that the left invariant vector field $X_{j}$ has an explicit form given by (see, e.g. \cite[Section 3.1.5]{FR16})
\begin{equation}
    \label{Xj}
X_{j}=\frac{\partial}{\partial x_{j}^{\prime}}+\sum_{l=2}^{r} \sum_{m=1}^{N_{l}} a_{j, m}^{(l)}\left(x^{\prime}, \ldots, x^{(l-1)}\right) \frac{\partial}{\partial x_{m}^{(l)}}.
\end{equation}
Throughout this section, we will also use the following notations
$$
\nabla_{H}:=\left(X_{1}, \ldots, X_{N_{1}}\right)
$$
for the horizontal gradient,
\begin{equation}
    \label{psublap}
\mathcal{L}_{p} f:=\nabla_{H}\left(\left|\nabla_{H} f\right|^{p-2} \nabla_{H} f\right), \quad 1<p<\infty,
\end{equation}
for the $p$-sub-Laplacian, and
$$
\left|x^{\prime}\right|=\sqrt{x_{1}^{\prime 2}+\cdots+x_{N_{1}}^{\prime 2}}
$$
for the Euclidean norm on $\mathbb{R}^{N_{1}}$.

Using \eqref{Xj} one can observe that (see, e.g. \cite{RS17})
\begin{equation}
    \label{3.3}
\left.\left.\left|\nabla_{H}\right| x^{\prime}\right|^{b}|=b| x^{\prime}\right|^{b-1},
\end{equation}
and
\begin{equation}
    \label{3.4}
\nabla_{H}\left(\frac{x^{\prime}}{\left|x^{\prime}\right|^{b}}\right)=\frac{N_{1}-b}{\left|x^{\prime}\right|^{b}}
\end{equation}
for all $b \in \mathbb{R}, x^{\prime} \in \mathbb{R}^{N_{1}}$ and $\left|x^{\prime}\right| \neq 0$.

Let us first consider the following initial boundary value problem for the $p$-sub-Laplacian, $1<p<\infty$,
\begin{equation}
    \label{psublap_pr_111}
\left\{\begin{array}{l}
{u_{t}-\mathcal{L}_{p} u=\alpha|u|^{q-1} u+\beta|\nabla_{H}u|^{r}, \quad x \in \Omega, \quad t>0,} \\
{u(t, x)=0, \quad x \in \partial \Omega, \quad t>0,} \\
{u(0, x)=u_{0}(x), \quad x \in \Omega,}
\end{array}\right.
\end{equation}
where $u_{0}(x) \geqslant 0$, $u_{0}(x) \not \equiv 0$, $u_{0} \in S_{0}^{1, \infty}(\Omega)$, and the parameters $\alpha$, $\beta$, $q$ and $r$ will be determined later. By Definition \ref{def_weak} let us recall that $T_{\max }$ is the
maximal existence time of a weak solution of \eqref{psublap_pr_111}.
\begin{thm}\label{glob_exis_3.5_111}  Let $\Omega \subset \G$ be a bounded open set in a stratified Lie group with $N_{1}$ being the dimension of the first stratum. Assume that $\alpha$, $\beta$, $p,q$ and $r$ in \eqref{psublap_pr_111} satisfy one of the following conditions:
\begin{enumerate}[label=(\roman*)]
    \item $\alpha<0$, $\beta>0$, and $p,r>1$ with $p\leq r+1$ and $p/2\leq r<q$;
    \item $\alpha>0$, $\beta<0$, and $p>1$, $q\geq 1$ with $p<r+1$ and $q\leq r$.
    \end{enumerate}
Then a weak solution of \eqref{psublap_pr_111} is globally in $t$-bounded, that is, there exists a constant $M$ depending only on $p$, $q$, $r$, $\alpha$, $\beta$, $N_{1}$, $\Omega$ and $u_{0}$ such that for every $T>0$ we have $0\leq u\leq M$ on $(0,T)$.
\end{thm}
\begin{proof}[Proof of Theorem \ref{glob_exis_3.5_111}] Part (i). For convenience, we assume that $\beta=-\alpha=1$. Set $R^{\prime}:=\underset{x=\left(x^{\prime}, x^{\prime \prime}\right) \in \Omega}{\rm max}\left|x^{\prime}\right|$. Then, since $\Omega$ is bounded, we get $R^{\prime}<\infty$. For any $x=\left(x^{\prime}, x^{\prime \prime}\right) \in \Omega$, let $x_{0}=\left(x_{0}^{\prime}, x_{0}^{\prime \prime}\right) \in \G \backslash \Omega$ and $\varepsilon \in(0,1)$ be such that $\varepsilon \leqslant\left|x_{0}^{\prime}-x^{\prime}\right|<R^{\prime}+1$. We also introduce the following notations
$$
V_{1}(t, x):=K_{1} e^{\sigma_{1} R}, \quad R=\left|x^{\prime}-x_{0}^{\prime}\right|, \quad x=\left(x^{\prime}, x^{\prime \prime}\right) \in \Omega,
$$
and $$
\mathcal{M}_{p} w:=w_{t}-\mathcal{L}_{p} w+w^{q}-|\nabla_{H} w|^{r}.
$$
Let us now find suitable positive $K_{1}$ and $\sigma_{1}$ such that $V_{1}(t, x)$ is a super-solution of \eqref{psublap_pr_111}. By using the identities \eqref{3.3} and \eqref{3.4}, we observe that
$$|\nabla_{H} V_{1}|^{r}=K_{1}^{r}e^{\sigma_{1} Rr}\sigma_{1}^{r}$$
and
\begin{equation*}
    \begin{split}
\mathcal{L}_{p} V_{1}&=\nabla_{H}\left(\left|\nabla_{H} K_{1} e^{\sigma_{1} R}\right|^{p-2} \nabla_{H} K_{1} e^{\sigma_{1} R}\right)\\&=\nabla_{H}\left(K_{1}^{p-2} \sigma_{1}^{p-2} e^{\sigma_{1}\left|x^{\prime}-x_{0}^{\prime}\right|(p-2)} K_{1} \sigma_{1} e^{\sigma_{1}\left|x^{\prime}-x_{0}^{\prime}\right|} \frac{x^{\prime}-x_{0}^{\prime}}{\left|x^{\prime}-x_{0}^{\prime}\right|}\right)\\&=
\nabla_{H}\left(K_{1}^{p-1} \sigma_{1}^{p-1} e^{\sigma_{1}\left|x^{\prime}-x_{0}^{\prime}\right|(p-1)} \frac{x^{\prime}-x_{0}^{\prime}}{\left|x^{\prime}-x_{0}^{\prime}\right|}\right)\\&=K_{1}^{p-1} \sigma_{1}^{p-1} \sigma_{1}(p-1) e^{\sigma_{1}\left|x^{\prime}-x_{0}^{\prime}\right|(p-1)}+K_{1}^{p-1} \sigma_{1}^{p-1} e^{\sigma_{1}\left|x^{\prime}-x_{0}^{\prime}\right|(p-1)} \frac{N_{1}-1}{\left|x^{\prime}-x_{0}^{\prime}\right|}\\&=(p-1) \sigma_{1}^{p} K_{1}^{p-1} e^{(p-1) \sigma_{1} R}+\frac{N_{1}-1}{R} \sigma_{1}^{p-1} K_{1}^{p-1} e^{\sigma_{1} R(p-1)}.
            \end{split}
\end{equation*}
Thus, we have
$$\mathcal{M}_{p} V_{1}=-(p-1) \sigma_{1}^{p} K_{1}^{p-1} e^{(p-1) \sigma_{1} R}-\frac{N_{1}-1}{R} \sigma_{1}^{p-1} K_{1}^{p-1} e^{(p-1) \sigma_{1} R}+K_{1}^{q} e^{q \sigma_{1} R}-K_{1}^{r} e^{r \sigma_{1} R}\sigma_{1}^{r}.$$
Now, we need to find $\sigma_{1}$ and $K_{1}$ such that $\mathcal{M}_{p} V_{1} \geqslant 0$, that is,
$$
(p-1) \sigma_{1}^{p} K_{1}^{p-1} e^{(p-1) \sigma_{1} R}+\frac{N_{1}-1}{R} \sigma_{1}^{p-1} K_{1}^{p-1} e^{(p-1) \sigma_{1} R}+K_{1}^{r} e^{r \sigma_{1} R}\sigma_{1}^{r} \leqslant K_{1}^{q} e^{q \sigma_{1} R}.
$$
Multiplying both sides of the inequality by $K_{1}^{-p+1} e^{-(p-1) \sigma_{1} R}$, we derive that
\begin{equation}
    \label{3.7}
(p-1) \sigma_{1}^{p}+\frac{N_{1}-1}{R} \sigma_{1}^{p-1}+K_{1}^{r-p+1} e^{(r-p+1) \sigma_{1} R}\sigma_{1}^{r} \leqslant K_{1}^{q+1-p} e^{(q+1-p) \sigma_{1} R}.
\end{equation}
Taking into account $\varepsilon \leqslant R<R^{\prime}+1$, we see that in order to prove \eqref{3.7} it is sufficient to show
$$
(p-1) \sigma_{1}^{p}+\frac{N_{1}-1}{\varepsilon} \sigma_{1}^{p-1}+K_{1}^{r-p+1} e^{(r-p+1) \sigma_{1} (R'+1)}\sigma_{1}^{r} \leqslant K_{1}^{q+1-p}.
$$
Thus, to have $\mathcal{M}_{p} V_{1} \geqslant 0$ we can choose
$$
\begin{array}{c}
{\sigma_{1}=\frac{1}{(r-p+1)\left(R^{\prime}+1\right)}}, \\
{K_{1}=\max \left\{(2 e\sigma_{1}^{r})^{1 /(q-r)},\left(2\left((p-1) \sigma_{1}^{p}+\frac{N_{1}-1}{\varepsilon} \sigma_{1}^{p-1}\right)\right)^{1 /(q+1-p)}\right\}}
\end{array}
$$
when $r+1>p$, and
$$
\sigma_{1}=1, \quad K_{1}=\max \left\{2^{1 /(q-r)},\left(2\left(p-1+\frac{N_{1}-1}{\varepsilon}\right)\right)^{1 /(q+1-p)}\right\}
$$
when $r+1=p$. We also need that $K_{1} \geqslant\left\|u_{0}\right\|_{L^{\infty}(\Omega)}$ such that $V_{1}(0, x)=K_{1} e^{\sigma_{1} R} \geqslant u_{0}(x)$. Obviously, we also have $V_{1}(t,x) \geq 0=u(t,x)$ on $\partial \Omega$. Therefore, $V_{1}(t, x)$ is a super-solution of \eqref{psublap_pr_111}. Then, Theorem \ref{compar_thm} concludes that
\begin{equation}
    \label{3.10_11}
0\leq u(t, x) \leqslant K_{1} e^{\sigma_{1}\left(R^{\prime}+1\right)}<\infty, \quad R^{\prime}=\max _{x=\left(x^{\prime}, x^{\prime \prime}\right) \in \Omega}\left|x^{\prime}\right|.
\end{equation}
Note that the right-hand side of \eqref{3.10_11} is independent of $t$, hence $u(t, x)$ is globally in $t$-bounded.

Part (ii). In this case, we may assume that $\alpha=-\beta=1$. We recall from Part (i) that $R^{\prime}=\underset{x=\left(x^{\prime}, x^{\prime \prime}\right) \in \Omega}{\rm max}\left|x^{\prime}\right|<\infty$ and $\varepsilon \leqslant\left|x_{0}^{\prime}-x^{\prime}\right|<R^{\prime}+1$ for any $x=\left(x^{\prime}, x^{\prime \prime}\right) \in \Omega$, where $x_{0}=\left(x_{0}^{\prime}, x_{0}^{\prime \prime}\right) \in \G \backslash \Omega$ and $\varepsilon \in(0,1)$.

First, let us consider the case $r>q$. Here, we will use the following notations $$V_{2}(t,x):=\frac{K_{2}}{\sigma_{2}} R^{\sigma_{2}}, \quad \sigma_{2}=\frac{p}{p-1}, \quad R=\left|x'-x'_{0}\right|, \quad x \in \Omega,$$
and $$
\mathcal{N}_{p} w:=w_{t}-\mathcal{L}_{p} w-w^{q}+|\nabla_{H} w|^{r}.
$$
Now, we need to find a suitable positive $K_{2}$ such that $V_{2}(t, x)$ is a super-solution of \eqref{psublap_pr_111}. By using the identities \eqref{3.3} and \eqref{3.4}, we observe that
\begin{equation*}
    \begin{split}
\mathcal{L}_{p} V_{2}&=\nabla_{H}\left(\left|\nabla_{H} \left(\frac{K_{2}}{\sigma_{2}} R^{\sigma_{2}}\right)\right|^{p-2} \nabla_{H} \left(\frac{K_{2}}{\sigma_{2}} R^{\sigma_{2}}\right)\right)\\&=\left(\frac{K_{2}}{\sigma_{2}}\right)^{p-1}\nabla_{H}\left(\sigma_{2}^{p-2} R^{(\sigma_{2}-1)(p-2)} \sigma_{2} R^{\sigma_{2}-1}\frac{x^{\prime}-x_{0}^{\prime}}{\left|x^{\prime}-x_{0}^{\prime}\right|}\right)\\&=
K_{2}^{p-1}\nabla_{H}\left( R^{(\sigma_{2}-1)(p-1)} \frac{x^{\prime}-x_{0}^{\prime}}{\left|x^{\prime}-x_{0}^{\prime}\right|}\right)\\&=N_{1}K_{2}^{p-1}.
            \end{split}
\end{equation*} Then we have
$$
\mathcal{N}_{p} V_{2}=-N_{1} K_{2}^{p-1}+K_{2}^{r} R^{\frac{r}{p-1}}-\left(\frac{K_{2}}{\sigma_{2}}\right)^{q} R^{\frac{q p}{p-1}}.
$$
From this, we have
$$
\mathcal{N}_{p} V_{2} \geq 0 \Longleftrightarrow K_{2}^{r} R^{\frac{r}{p-1}} \geq N_{1} K_{2}^{p-1}+\left(\frac{K_{2}}{\sigma_{2}}\right)^{q} R^{\frac{qp}{p-1}}.
$$
Thus, it is sufficient to choose $K_{2}$ such that
\begin{equation}
    \label{3.20_11}
    K_{2}^{r} R^{\frac{r}{p-1}} \geq 2 N_{1} K_{2}^{p-1},
\end{equation}
\begin{equation}
    \label{3.21_11}
K_{2}^{r} R ^{\frac{r}{p-1}}  \geq 2\left(\frac{K_{2}}{\sigma_{2}}\right)^{q} R^{\frac{qp}{p-1}}.
\end{equation}
Note that the inequality \eqref{3.20_11} is satisfied if we take
$$
K_{2} \geq\left(\frac{2 N_{1}}{\varepsilon^{\frac{r}{p-1}}}\right)^{\frac{1}{r-p+1}},
$$
provided that $r>p-1$. We divide inequality \eqref{3.21_11} by $K_{2}^{q} R^{ \frac{r}{p-1}}$ to derive
$$
K_{2}^{r-q} \geq \frac{2}{\sigma_{2}^{q}} R^{\frac{q p-r}{p-1}}.
$$
For $q p \geq r$, we can set
$$
K_{2} \geq\left(\frac{2}{\sigma_{2}^{q}}\right)^{\frac{1}{r-q}}(R'+1)^{\frac{q p-r}{(p-1)(r-q)}},
$$
while for $q p<r$, we can set
$$
K_{2} \geq\left(\frac{2}{\sigma_{2}^{q}}\right)^{\frac{1}{r-q}} \varepsilon^{\frac{qp-r}{(p-1)(r-q)}}.
$$
We also need that $K_{2} \geq \frac{\sigma_{2} \|u_{0}\|_{L^{\infty}}}{\varepsilon^{\sigma_{2}}}$ to have $V_{2}(0, x) \geq u_{0}$. Thus, taking $K_{2}$ as follows
\begin{equation*}
K_{2} \geq \max \left\{\frac{\sigma_{2}\left\|u_{0}\right\|_{L^ \infty}}{\varepsilon^{\sigma_{2}}},\left(\frac{2 N_{1}}{\varepsilon^{\frac{r}{p-1}}}\right)^{\frac{1}{r-p+1}},\left(\frac{2}{\sigma_{2}^{q}}\right)^{\frac{1}{r-q}}(R'+1)^{\frac{qp-r}{(p-1)(r-q)}},
\left(\frac{2}{\sigma_{2}^{q}}\right)^{\frac{1}{r-q}} \varepsilon^{\frac{q p-r}{(p-1)(r-q)}}\right\},
\end{equation*}
we obtain $\mathcal{N}_{p}V_{2}\geq 0$ and $V_{2}(0, x) \geq u_{0}$. It is clear that $V_{2}(t,x) \geq 0=u(t,x)$ on $\partial \Omega$. Therefore, $V_{2}(t,x)$ is a super-solution of \eqref{psublap_pr_111}. Then, Theorem \ref{compar_thm} concludes that
$$
0\leq u(t,x) \leq \frac{K_{2}(R'+1)^{\frac{p}{p-1}}}{\sigma_{2}}<\infty.
$$
In the case when $r=q$, we can take
$$\sigma_{3} \geq \max \left\{1,2^{1 / r}(R'+1)\right\},$$
and
$$
K_{3} \geq \max \left\{\varepsilon^{-\sigma_{3}}\left\|u_{0}\right\|_{L^ \infty},\left(\frac{2((p-1)(\sigma_{3}-1)+N_{1}-1)}{\varepsilon^{(r-p+1)(\sigma_{3}-1)+1}}\right)^{\frac{1}{r-p+1}}\right\},
$$
such that the function $V_{3}(t, x)=K_{3} R^{\sigma_{3}}$ is a super-solution of \eqref{psublap_pr_111}. By the same procedure, one can obtain the uniform boundedness of $u(t, x)$.
\end{proof}

\begin{thm}\label{thm_1.3} Let $\alpha>0$, $\beta<0$, $p>1$ and $r>0$. If $q>\max \{p-1, r, 1\}$, then the solution of the problem \eqref{psublap_pr_111} blows up in finite time for some large $u_{0}>0$.
\end{thm}
\begin{proof}[Proof of Theorem \ref{thm_1.3}]
For convenience, let us assume that $\alpha=-\beta=1$. Set
\begin{equation}\label{v_func}
v(t,|x'|):=\frac{1}{(1-\delta t)^{k_{1}}} F\left(\frac{|x'|}{(1-\delta t)^{k_{2}}}\right), \quad t_{0} \leq t<\frac{1}{\delta},
\end{equation}
where
$$
F(y):=1+\frac{A}{\sigma}-\frac{y^{\sigma}}{\sigma A^{\sigma-1}}, \quad y \geq 0, \quad \sigma=\frac{p}{p-1},
$$
and
\begin{equation}\label{range_par}
k_{1}=\frac{1}{q-1}, \quad 0<k_{2}<\min \left\{\frac{q-p+1}{p(q-1)}, \frac{q-r}{r(q-1)}\right\}, \quad A>\frac{k_{1}}{k_{2}}, \quad \delta<\frac{1}{k_{1}\left(1+\frac{A}{\sigma}\right)}.
\end{equation}
Then, it can be noted that $v(t,|x'|)$ is positive and smooth when $t\in [t_{0},\frac{1}{\delta})$ and $|x'|<R_{1}(1-\delta t)^{k_{2}}$, where $R_{1}:=\left(A^{\sigma-1}(A+\sigma)\right)^{1 / \sigma}$.

We want to show that $v(t, |x'|)$ is a sub-solution of \eqref{psublap_pr_111}. For $y=\frac{|x'|}{(1-\delta t)^{k_{2}}}$, by a direct calculation we have
$$
\begin{aligned}
\mathcal{N}_{p} v=& v_{t}-\mathcal{L}_{p} v-v^{q}+|\nabla_{H} v|^{r}\\&=\frac{\delta\left(k_{1} F+k_{2} y F^{\prime}\right)}{(1-\delta t)^{k_{1}+1}}-\frac{\left(\left|F^{\prime}\right|^{p-2} F^{\prime}\right)^{\prime}+\frac{N_{1}-1}{y}\left|F^{\prime}\right|^{p-2} F^{\prime}}{(1-\delta t)^{(p-2)(k_{1}+k_{2})+(k_{1}+2 k_{2})}} \\
&-\frac{F^{q}}{(1-\delta t)^{q k_{1}}}+\frac{\left|F^{\prime}\right|^{r}}{(1-\delta t)^{r(k_{1}+k_{2})}}.
\end{aligned}
$$
Note that $k_{1}q=k_{1}+1>(p-2)(k_{1}+k_{2})+k_{1}+2k_{2}$ and $k_{1}+1>(k_{1}+k_{2}) r$ by \eqref{range_par}. Observe that
$$\left(\left|F^{\prime}\right|^{p-2} F^{\prime}\right)^{\prime}+\frac{N_{1}-1}{y}\left|F^{\prime}\right|^{p-2} F^{\prime}=-\frac{N_{1}}{A}, \quad 0<y<R_{1}.$$

Let us now show that $\mathcal{N}_{p} v\leq 0$ for all $t\in [t_{0},\frac{1}{\delta})$ and $0 \leq y \leq R_{1}$. In the case $0 \leq y \leq A$, from the representation of $F(y)$ we note that
$$1 \leq F(y) \leq 1+\frac{A}{\sigma} \quad \text{and} \quad -1 \leq F^{\prime}(y) \leq 0.$$
Then, we can take $t_{0}=t_{0}(p, q, r, \delta, N_{1}, A)$ close to $\frac{1}{\delta}$ such that
\begin{multline}\label{3.5}
    \mathcal{N}_{p} v \leq \frac{1}{(1-\delta t)^{k_{1}+1}}\left(\delta k_{1}\left(1+\frac{A}{\sigma}\right)-1+\frac{N_{1}}{A}\left(1-\delta t_{0}\right)^{1-2 k_{2}-(p-2)(k_{1}+k_{2})}\right.\\\left.+\left(1-\delta t_{0}\right)^{k_{1}+1-r(k_{1}+k_{2})}\right) \leq 0.
\end{multline}
In the case $A \leq y \leq R_{1}$, we have
$$0 \leq F(y) \leq 1 \quad \text{and} \quad -\left(\frac{R_{1}}{A}\right)^{\sigma-1} \leq F^{\prime}(y) \leq-1.$$
Similarly as above, one verifies that
\begin{multline}
    \label{3.6}
\mathcal{N}_{p} v \leq  \frac{1}{(1-\delta t)^{k_{1}+1}}\left(\delta(k_{1}-k_{2} A)+\frac{N_{1}}{A}\left(1-\delta t_{0}\right)^{1-2 k_{2}-(p-2)(k_{1}+k_{2})}\right.\\\left.
+\left(\frac{R_{1}}{A}\right)^{r(\sigma-1)}\left(1-\delta t_{0}\right)^{k_{1}+1-r(k_{1}+k_{2})}\right) \leq 0.
\end{multline}
From \eqref{3.5} and \eqref{3.6}, we conclude that $\mathcal{N}_{p} v \leq 0$ for all $t\in [t_{0},\frac{1}{\delta})$ and $|x'|<R_{1}(1-\delta t)^{k_{2}}$.

Next, we estimate $u_{0}$. By the group translation, without loss of generality we may assume that $\Omega$ contains the unit element of the group $\G$. Then, we can take suitable $t_{0}$ such that $R_{1}(1-\delta t_{0})^{k_{2}}<\underset{x=\left(x^{\prime}, x^{\prime \prime}\right) \in \Omega}{\rm max}\left|x^{\prime}\right|$ and $u_{0} \geq v\left(t_{0}, \cdot \right)$ in $\Omega\cap \{x=(x',x''):|x'|<R_{1}(1-\delta t_{0})^{k_{2}}\}$ for some large $u_{0}>0$. Then, taking into account $v\leq 0$ when $|x'|\geq R_{1}(1-\delta t)^{k_{2}}$, we obtain that $u_{0} \geq v\left(t_{0}, \cdot\right)$ in $\overline{\Omega}$. Obviously, we also have $v \leq 0$ when $(t,x) \in \left(t_{0}, \frac{1}{\delta}\right) \times \partial\Omega.$ Thus, the comparison principle (Theorem \ref{compar_thm}) implies that
$$
u(t,x) \geq v\left(t+t_{0}, x \right), \quad t\in \left[t_{0},\frac{1}{\delta}\right), \quad |x'|<R_{1}(1-\delta t)^{k_{2}}.
$$
On the other hand, by the definition of $v$ we have $\underset{t \rightarrow 1 / \delta}{\rm lim}v(t,0) \rightarrow \infty$. Consequently, $u$ must blow up at a finite time $T \leq \frac{1}{\delta}-t_{0}<\infty$.
\end{proof}
\begin{thm}\label{thm_1.4}
Assume that $\alpha<0$, $\beta>0$, $p,r>1$ and $q>0$ in \eqref{psublap_pr_111} satisfy one of the following conditions:
\begin{itemize}
    \item $r>\max \{p, q\}$;
    \item $r=q>p$, and $\beta \gg|\alpha|$.
    \end{itemize}
There exists $M>0$ such that if $\int_{\Omega} u_{0}^{\frac{2r-p}{r-p}} dx>M$, then $T_{max}<\infty$.
\end{thm}
\begin{proof}[Proof of Theorem \ref{thm_1.4}] Assume for a contradiction that $T_{\max }=\infty$. By $C_{1}$ and $C_{2}$ we denote positive constants which may vary from line to line. Set $\kappa=r /(r-p)$ and $y(t)=\frac{1}{\kappa+1} \int_{\Omega} u^{\kappa+1} dx$. Then, using $\kappa-1=\frac{p}{r-p}=\frac{p \kappa}{r}$, we have
$$
\begin{aligned}
y^{\prime}(t) &=\beta \int_{\Omega} u^{\kappa}|\nabla_{H} u|^{r} dx-\kappa \int_{\Omega} u^{\kappa-1}|\nabla_{H} u|^{p} dx-|\alpha| \int_{\Omega} u^{q+\kappa} dx \\
&=\beta \int_{\Omega} u^{\kappa}|\nabla_{H} u|^{r} dx-\kappa \int_{\Omega}\left(u^{\kappa}|\nabla_{H} u|^{r}\right)^{p / r} dx-|\alpha| \int_{\Omega} u^{q+\kappa} dx.
\end{aligned}
$$
For $r>q$, using H\"{o}lder's and Young's inequalities we get
$$
\begin{aligned}
\int_{\Omega}\left(u^{\kappa}|\nabla_{H} u|^{r}\right)^{p / r} dx & \leq\left(\int_{\Omega} u^{\kappa}|\nabla_{H} u|^{r} dx\right)^{p / r}|\Omega|^{(r-p) / r} \\
& \leq \epsilon \frac{p}{r} \int_{\Omega} u^{\kappa}|\nabla_{H} u|^{r} dx+C(\epsilon) \frac{r-p}{r}|\Omega|
\end{aligned}
$$
and
\begin{equation*}
    \begin{split}
\int_{\Omega} u^{q+\kappa} dx&=\int_{\Omega}\left(u^{r+\kappa}\right)^{\frac{q+\kappa}{r+\kappa}} dx \leq\left(\int_{\Omega} u^{r+\kappa} dx\right)^{\frac{q+\kappa}{r+\kappa}}|\Omega|^{\frac{r-q}{r+\kappa}}\\&
\leq \varepsilon \frac{q+\kappa}{r+\kappa} \int_{\Omega} u^{r+\kappa} dx+C(\varepsilon) \frac{r-q}{r+\kappa}|\Omega|.
    \end{split}
\end{equation*}
Then, by Poincar\'{e}'s (see, e.g. \cite[Formula 1.10]{RS17}) and reverse H\"{o}lder's inequalities we obtain
$$
\begin{aligned}
y^{\prime}(t) & \geq \frac{\beta r-\epsilon p}{r} \int_{\Omega} u^{\kappa}|\nabla_{H} u|^{r} dx-|\alpha| \varepsilon \frac{q+\kappa}{r+\kappa} \int_{\Omega} u^{r+\kappa} dx-C \\
&=\frac{\beta r-\epsilon p}{r}\left(\frac{r}{r+\kappa}\right)^{r} \int_{\Omega}|\nabla_{H} u ^{\frac{r+\kappa}{r}}|^{r} dx-|\alpha| \varepsilon \frac{q+\kappa}{r+\kappa} \int_{\Omega} u^{r+\kappa} dx-C \\
& \geq\left(\frac{\beta r-\epsilon p}{r}\left(\frac{r}{r+\kappa}\right)^{r} C^{\prime}-|\alpha| \varepsilon \frac{q+\kappa}{r+\kappa}\right) \int_{\Omega} u^{r+\kappa} dx-C \\
&=C_{1} \int_{\Omega} u^{r+\kappa} dx-C\\&
 \geq C_{1}\left(\int_{\Omega} u^{\kappa+1} dx\right)^{\frac{r+\kappa}{\kappa+1}}|\Omega|^{\frac{1-r}{\kappa+1}}-C_{2}\\&
 \geq C_{1}\left(\int_{\Omega} u^{\kappa+1} dx\right)^{\frac{r+\kappa}{\kappa+1}}-C_{2}.
\end{aligned}
$$
Thus, we have obtained
$$
y^{\prime}(t) \geq C_{1} y^{\frac{r+\kappa}{\kappa+1}}(t)-C_{2},
$$
where $C_{1}=C_{1}(p, r, q, \alpha, \beta, \epsilon, \varepsilon, \Omega, N_{1})$ and $C_{2}=C_{2}(p, r, q, \alpha, \beta, \epsilon, \varepsilon, \Omega, N_{1})>0$ with suitable $\epsilon$ and $\varepsilon$, and $N_{1}$ is the dimension of the first stratum of the group $\G$. Set
$$
M>\left(\frac{2 C_{2}}{C_{1}}\right)^{\frac{\kappa+1}{r+\kappa}},
$$
then if $y(0)>M$, we have
\begin{equation}
    \label{3.15}
y^{\prime}(t) \geq \frac{C_{1} y^{\frac{r+\kappa}{\kappa+1}}(t)}{2}.
\end{equation}
A contradiction then follows by integrating \eqref{3.15}, hence $T_{\max }<\infty$.

In the case $r=q$, the proof above is still valid for $\beta \gg|\alpha|$.
\end{proof}

As another application of the comparison principle, we now investigate the following initial boundary value problem for the $p$-sub-Laplacian, $1<p<\infty$,
\begin{equation}
    \label{psublap_pr}
\left\{\begin{array}{l}
{u_{t}-\mathcal{L}_{p} u=\alpha\sum_{i=1}^{n_{1}}|u|^{q_{i}-1}u+\gamma \sum_{i=1}^{n_{1}}|u|^{s_{i}-1}u, \quad x \in \Omega, \quad t>0,} \\
{u(t, x)=0, \quad x \in \partial \Omega, \quad t>0,} \\
{u(0, x)=u_{0}(x), \quad x \in \Omega,}
\end{array}\right.
\end{equation}
where $u_{0}(x) \geqslant 0$, $u_{0}(x) \not \equiv 0$, $u_{0} \in S_{0}^{1, \infty}(\Omega)$, and the parameters $\alpha$, $\gamma$, $q_{i}$ and $s_{i}$ will be determined later.
\begin{thm}\label{glob_exis_3.5}  Let $\Omega \subset \G$ be a bounded open set in a stratified Lie group with $N_{1}$ being the dimension of the first stratum. Let $\widetilde{s}=\min\{s_{i}\}$ and $\widetilde{q}=\min\{q_{i}\}$. Assume that $\alpha$, $\gamma$, $q_{i}$ and $s_{i}$ in \eqref{psublap_pr} satisfy one of the following conditions:
\begin{enumerate}[label=(\roman*)]
\item $\alpha>0$, $\gamma<0$, and $q_{i}\geq 1$ with $1<p<\widetilde{s}+1$ and $s_{i}<q_{i}$;
\item $\alpha<0$, $\gamma>0$, and $s_{i}\geq 1$ with $1<p<\widetilde{q}+1$ and $s_{i}>q_{i}$.
    \end{enumerate}
Then a weak solution of \eqref{psublap_pr} is globally in $t$-bounded, that is, there exists a constant $M$ depending only on $p$, $q_{i}$, $s_{i}$, $\alpha$, $\gamma$, $N_{1}$, $\Omega$ and $u_{0}$ such that for every $T>0$ we have $0\leq u\leq M$ on $(0,T)$.
\end{thm}
\begin{rem}
We refer to \cite[Section 3]{RS18} for a similar investigation when $\alpha=-\gamma=1$, $n_{1}=1$.
\end{rem}
\begin{proof}[Proof of Theorem \ref{glob_exis_3.5}] We only prove Part (i), since Part (ii) is actually the same, but only $\alpha$ and $q_{i}$ are swapped by $\beta$ and $s_{i}$, respectively. For convenience, we assume that $\alpha=-\gamma=1$. We recall that $R^{\prime}=\underset{x=\left(x^{\prime}, x^{\prime \prime}\right) \in \Omega}{\rm max}\left|x^{\prime}\right|<\infty$ and $\varepsilon \leqslant\left|x_{0}^{\prime}-x^{\prime}\right|<R^{\prime}+1$ for any $x=\left(x^{\prime}, x^{\prime \prime}\right) \in \Omega$, where $x_{0}=\left(x_{0}^{\prime}, x_{0}^{\prime \prime}\right) \in \G \backslash \Omega$ and $\varepsilon \in(0,1)$. We also employ the following notations
$$
V_{4}(t, x):=\frac{K_{4}}{\sigma_{4}} R^{\sigma_{4}}, \quad \sigma_{4}=\frac{p}{p-1}, \quad R=\left|x^{\prime}-x_{0}^{\prime}\right|, \quad x=\left(x^{\prime}, x^{\prime \prime}\right) \in \Omega,
$$
and
$$
\mathcal{K}_{p} w:=w_{t}-\mathcal{L}_{p} w-\sum_{i=1}^{n_{1}}w^{q_{i}}+\sum_{i=1}^{n_{1}}w^{s_{i}}.
$$
Now, we look for a suitable positive $K_{4}$ such that $V_{4}(t, x)$ is a super-solution of \eqref{psublap_pr}. Then we have
$$
\mathcal{K}_{p} V_{4}=-N_{1} K_{4}^{p-1}-\sum_{i=1}^{n_{1}}\left(\frac{K_{4}}{\sigma_{4}}\right)^{q_{i}} R^{\frac{q_{i} p}{p-1}}+\sum_{i=1}^{n_{1}}\left(\frac{K_{4}}{\sigma_{4}}\right)^{s_{i}} R^{\frac{s_{i} p}{p-1}}.
$$
From this, we note that
$$
\mathcal{K}_{p} V_{4} \geq 0 \Longleftrightarrow \sum_{i=1}^{n_{1}}\left(\frac{K_{4}}{\sigma_{4}}\right)^{s_{i}} R^{\frac{s_{i} p}{p-1}} \geq N_{1} K_{4}^{p-1}+\sum_{i=1}^{n_{1}}\left(\frac{K_{4}}{\sigma_{4}}\right)^{q_{i}} R^{\frac{q_{i} p}{p-1}}.
$$
So, it is sufficient to choose $K_{4}$ such that
\begin{equation}
    \label{3.20}
    \sum_{i=1}^{n_{1}}\left(\frac{K_{4}}{\sigma_{4}}\right)^{s_{i}} R^{\frac{s_{i} p}{p-1}} \geq 2 N_{1} K_{4}^{p-1},
\end{equation}
\begin{equation}
    \label{3.21}
\left(\frac{K_{4}}{\sigma_{4}}\right)^{s_{i}} R^{\frac{s_{i} p}{p-1}}  \geq 2\left(\frac{K_{4}}{\sigma_{4}}\right)^{q_{i}} R^{\frac{q_{i} p}{p-1}}.
\end{equation}
The inequality \eqref{3.20} is satisfied if we take
$$
K_{4} \geq\left(2 N_{1}\right)^{\frac{1}{\widetilde{s}-p+1}} \left(\sum_{i=1}^{n_{1}}\frac{\varepsilon^{\frac{s_{i}p}{p-1}}}{\sigma_{4}^{s_{i}}}\right)^{-\frac{1}{\widetilde{s}-p+1}},
$$
provided that $\widetilde{s}=\min\{s_{i}\}>p-1$. Dividing the inequality \eqref{3.21} by $K_{4}^{q_{i}} \frac{R^{ \frac{s_{i}p}{p-1}}}{\sigma_{4}^{s_{i}}}$ we get
$$
K_{4}^{s_{i}-q_{i}} \geq 2\sigma_{4}^{s_{i}-q_{i}} R^{\frac{p(q_{i}-s_{i})}{p-1}},
$$
that is,
$$
K_{4} \geq 2\sigma_{4}\varepsilon^{-\frac{p}{p-1}}.
$$
We also need that $K_{4} \geq \frac{\sigma_{4} \|u_{0}\|_{L^{\infty}}}{\varepsilon^{\sigma_{4}}}$ to ensure $V_{4}(0, x) \geq u_{0}$. Thus, choosing $K_{4}$ as follows
\begin{equation*}
 K_{4} \geq \max \left\{\frac{\sigma_{4}\left\|u_{0}\right\|_{L^ \infty}}{\varepsilon^{\sigma_{4}}},\left(2 N_{1}\right)^{\frac{1}{\widetilde{s}-p+1}} \left(\sum_{i=1}^{n_{1}}\frac{\varepsilon^{\frac{s_{i}p}{p-1}}}{\sigma_{4}^{s_{i}}}\right)^{-\frac{1}{\widetilde{s}-p+1}}, 2\sigma_{4}\varepsilon^{-\frac{p}{p-1}}\right\},
\end{equation*}
we obtain $\mathcal{K}_{p}V_{4}\geq 0$ and $V_{4}(0, x) \geq u_{0}$. Clearly, we also have $V_{4}(t,x) \geq 0=u(t,x)$ on $\partial \Omega$. Therefore, we can conclude that $V_{4}(t,x)$ is a super-solution of \eqref{psublap_pr}. Then, the comparison principle yields that
\begin{equation}\label{3.10}
0\leq u(t,x) \leq \frac{K_{4}(R'+1)^{\frac{p}{p-1}}}{\sigma_{4}}<\infty, \quad R^{\prime}=\max _{x=\left(x^{\prime}, x^{\prime \prime}\right) \in \Omega}\left|x^{\prime}\right|.
\end{equation}
Since the right-hand side of \eqref{3.10} is independent of $t$, we can conclude that $u(t, x)$ is globally in $t$-bounded.
\end{proof}

By the same procedure as in the proof of Theorem \ref{thm_1.3}, one can obtain the following result for the problem \eqref{psublap_pr} when $n_{1}=1$:
\begin{thm}\label{anoth_eq_appl_thm} Let $\alpha>0$, $\gamma<0$, $p>1$ and $s>0$. If $q>\max \{s, p-1, 1\}$, then the solution of the problem \eqref{psublap_pr} blows up in finite time for some large $u_{0}>0$.
\end{thm}
\begin{proof}[Proof of Theorem \ref{anoth_eq_appl_thm}] As in the proof of Theorem \ref{thm_1.3}, one can show that the same function $v$ from \eqref{v_func} is a sub-solution of the problem \eqref{psublap_pr}. Then, the comparison principle (Theorem \ref{compar_thm}) concludes the proof.
\end{proof}

\end{document}